\documentclass[11pt,a4paper]{article}
\usepackage{graphicx}
\usepackage{amsmath}
\usepackage{amsmath, amsthm, amscd, amsfonts, amssymb}

\usepackage{subfigure}

\newtheorem{theorem}{Theorem}

\newtheorem{definition}[theorem]{Definition}

\newtheorem{lemma}[theorem]{Lemma}

\newtheorem{remark}[theorem]{Remark}

\newcommand{\eps}{\varepsilon}
\renewcommand{\epsilon}{\eps}
\newcommand{\N}{\mathbb{N}}

\begin{document}

\allowdisplaybreaks

\title{Numerically Computable A Posteriori-Bounds for stochastic Allen-Cahn equation}

\author{Dirk Bl\"omker\footnote{Institut f\"ur Mathematik Universit\"at Augsburg, 86135 Augsburg, Germany,
{\sc e-mail:} dirk.bloemker@math.uni-augsburg.de},
Minoo Kamrani\footnote{Department of Mathematics, Faculty of Sciences,
	Razi University, Kermanshah, Iran}}

\date{\today}

\maketitle

\begin{abstract}
	The aim of this paper is the derivation of an a-posteriori
	error estimate for a numerical method based on an exponential scheme in time 
	and spectral Galerkin methods in space.
	We  obtain analytically a rigorous bound on
	the mean square error conditioned to the calculated data,
	which is numerically computable and uses the given numerical approximation.
	Thus one can check  a-posteriori the error for a given numerical computation without relying on an asymptotic result.
	
	All estimates are only based on the numerical data and the structure of the equation,
	but they do not use any a-priori information of the solution, which makes the approach applicable to equations where global existence of solutions is not known.
	For simplicity of presentation, we develop the method here in a relatively simple situation of
	a stable one-dimensional Allen-Cahn equation with additive forcing.
\end{abstract}


\section{Introduction}

A-posteriori analysis of deterministic PDE (partial differential equations) is a well developed tool.
See for example the book \cite{Ver:13} or the results for Allen-Cahn and related equations \cite{BaMuOr:12,Ba:05,GeMa:14}.
The strength of the method is usually the derivation of error indicators for the
refinement of meshes in adaptive schemes. See \cite{SSST:05} for an example in a stochastic setting.

Also for SPDEs (stochastic PDEs) there are recent results on a-posteriori analysis.
The results of \cite{BuDaWi:11, KaFrYa:13} use a-posteriori estimates in polynomial or Wiener-chaos expansion,
and the results of \cite{YaDuGu:10, YaQiDu:12} show a-posteriori mean square error estimates,
under the assumption that  the whole law of the  numerical approximation is known 
(or at least several moments of it).

In our work we follow a different more path-wise approach.
We measure the error in mean square conditioned on the calculated numerical data.
Given a single realization of the numerical approximation, without using a-priori information on the solution
we show analytic bounds, that can be calculated numerically, and guarantee a-posteriori that the true solution
is close to the given realization of the numerical approximation, which was calculated.

Let us remark that the mean square error of our approximation scheme might diverge (see Jentzen \& Hutzenthaler \cite{HuJe1,HuJe2}).
Thus it is not obvious that our conditional mean square error converges, although we obtain a good
error estimate in our numerical example. 
Moreover, we expect quite a large variation for different numerical realizations, 
which seems to be also visible in our numerical examples.

The general philosophy of a-priori error analysis is to use the true solution,
which is plugged into the numerical scheme to calculate the residual. 
Then using the discrete in time equation given by the numerical scheme,
one can derive a discrete equation for the error,
which has coefficients depending on the true solution.
Using a-priori information of the solution, asymptotic bounds for the error are derived.

In our a-posteriori analysis we use a time-continuous interpolation of the numerical data, which is plugged into the SPDE,
in order to derive bounds on the residual. For the error we obtain a PDE which is continuous in
time and has coefficients depending on the numerical data.
Here we can use now standard a-priori SPDE-type methods to derive error bounds,
that depend only on numerical data and the residual, which can be calculated rigorously from the numercial data.

Although for simplicity of presentation, we use a much simpler equation of Allen-Cahn-type,
our result is motivated by equations where the global existence of solutions is not known,
and thus global a-priori estimates are not available.
Typical examples are the three-dimensional Navier-Stokes equation or a somewhat simpler equation from surface growth \cite{BlRo:15}.
For the latter in \cite{BlNoRo:15, No:17} a-posteriori analysis was used for the deterministic PDEs to prove numerically the regularity of solutions
and thus the global  existence and uniqueness.

Here we focus as a starting point for simplicity on an one-dimensional equation of Allen-Cahn type.
Here even the asymptotic convergence results of numerical schemes are well known
See for example \cite{LPS:14,KLMS:11} or \cite{BJ:16P} for a truncated scheme.
Moreover, there is no problem with existence  and uniqueness of solutions. See for example \cite{PrZa:14}.

For the spatial discretization we use the spectral Galerkin-scheme, which simplifies the analysis.
Moreover, for the time-discretization we use a variant of the exponential scheme introduced by \cite{JeKlWi:11}.
Asymptotically, both exponential discretization schems should be equivalent, but the variant we use is slightly easier to handle in the analysis.

The precise functional analytic set-up and the equation itself is presented in Section \ref{sec:sett}.
In Section \ref{sec:SB} we present analytic results for stochastic terms which we cannot evaluate numerically.
One is the infinite-dimensional remainder of the stochastic convolution at discretization times.
The second one bounds fluctuations in between discretization times. Here we need to analyze an Ornstein-Uhlenbeck bridge-process,
as we now the stochastic convolution at all discretization times.

In the main result we present in Section \ref{sec:Res} analytic error estimates for the residual that depend only on the numerically calculated data,
the initial condition, and the stochastic terms already bounded in Section \ref{sec:SB}. We provide a bound in  moments of the
$L^4$-norm, which is conditioned on the given numerical data.
In  Section \ref{sec:App} we study the conditional mean square error of the approximation in $L^2$-norm,
given the numerical data. Nevertheless, this is a property of the equation and not of the data.
We need to quantify the continuous dependence of solutions on additive perturbations like the stochastic convolution or the residual.
Due to the relatively simple structure of the equation with a stable nonlinearity and a stable linear part,
this is relatively straightforward.

In the final section, we give numerical examples to illustrate the result. 
Here we use a quite poor discretization given that the solution is very rough
and still obtain meaningful error bounds. In more detail we study a finer discretization, 
where we see that the rigorous error estimate bounds the solution well.
One main source of error comes from bounds on terms that appear due to stochastic fluctuations 
between the discretization points
and not by the error  at the discrete times where the approximation is calculated.


\section{Setting}
\label{sec:sett}

The following assumptions and definitions are used throughout the paper.
Consider the following SPDE on the Hilbert-space $H=L^2([0,\pi])$  which is of the type:
\begin{equation}\label{SPDE}
du = [Au + F(u)] dt + dW\; \qquad u(0) = u_\star\;,
\end{equation}
subject to Dirichlet boundary conditions on $[0,\pi]$,
where $A$ is the Laplacian, $W$ some cylindrical $Q$-Wiener process.
Finally, $F$ is the locally-Lipschitz nonlinearity $F(u)=-u^3$.

The Dirichlet Laplacian  $A$ is diagonal w.r.t.\ $e_k(x)=\sqrt{2/\pi}\sin(kx)$, $k\in\mathbb{N}$
and generates an analytic semigroup  $\{e^{tA}\}_{t\geq 0}$ on $H$.
Moreover, it is a contraction semigroup on any $L^p(0,\pi)$.
This follows in $L^2$ as the largest eigenvalue of $A$ is $-1$ and thus $\|e^{tA}\|_{\mathcal{L}(L^2)}\leq e^{-t}$.
In $L^\infty$ it is true by the maximum principle $\|e^{tA}\|_{\mathcal{L}(L^\infty)}\leq 1$.
Then by the Riesz-Thorin theorem for any $L^p$-space we have for $t>0$
\begin{equation}
\label{e:RT}
\|e^{tA}\|_{\mathcal{L}(L^p)}
\leq \|e^{tA}\|_{\mathcal{L}(L^\infty)}^{(p-2)/p} \|e^{tA}\|_{\mathcal{L}(L^2)}^{2/p}
\leq e^{-2t/p } < 1 \;.
\end{equation}
Let us remark that by geometric series $I-e^{tA}$ is an invertible operator in $L^p$ with bounded inverse.

For simplicity we assume that the covariance operator $Q$ is also  diagonal in the Fourier basis $e_k$,
and denote the eigenvalues by $\alpha_k^2$, i.e.\ $Qe_k=\alpha_k^2e_k$.
This is a standard assumption in order to simplify the analysis by working with explicit Fourier series.
We comment later on possible generalizations.

In our numerical examples we consider space-time-white noise of order one that
corresponds to $\alpha_k=1$ for all $k\in\mathbb{N}$, which is a case of  solutions of quite poor regularity  with strong fluctuations.
Nevertheless we could allow for even rougher noise.

We assume that $\sum_{k\in\mathbb{N}}\alpha_k^2k^{-2+\delta}<\infty$ for some $\delta>0$, which guarantees that
the stochastic convolution (or Ornstein-Uhlenbeck process) 
\[
Z(t)=\int_0^t e^{(t-s)A} dW(s) \;,
\]
is continuous both in space and time (cf. \cite{PrZa:14}).

The mild formulation of  (\ref{SPDE}) is defined by the fixed point equation
\begin{equation}
 \label{e:mild}
 u(t) =e^{At}u_\star + \int_{0}^{t}e^{A(t-s)}F(u(s))ds + Z(t)\;.
\end{equation}
Here
the existence and uniqueness of solutions is standard. See for example \cite{PrZa:14}.
Here we just need  for some $p\geq3$ that $u \in C^0([0,T], L^p)$, $\mathbb{P}$-almost sure,
in order to formulate the mild solution in $L^p$ and  apply fixed point theorem to \eqref{e:mild}.
For the rest of the paper we just assume that $u_\star$ is such that
there is a sufficiently smooth unique mild solution.

\subsection{Discretization}
\label{sec:defphi}
Here we define the discretization scheme used throughout the paper.
For the discretization in space we use a spectral Galerkin method. Define $H_N$ as the $N$-dimensional space spanned
by the first $N$ eigenfunctions $e_1,\ldots,e_N$. Moreover, denote the orthogonal projection onto $H_N$ as $P_N$.

For the discretization in time, we use a fixed step-size $h = T/M >0$ and for a fixed realization $\omega$,
using a random number generator, we obtain in principle exact values of
\[
\{ P_N Z(hk) \}_{k\in\N} \;,
\]
defined by
\[
Z_0 = 0\;,
\qquad
Z_{k+1}
=  e^{hA} Z_k + X_{k+1}
=  \sum_{j=1}^{k+1} e^{h(k+1-j)A} P_N Q^{1/2} X_j
\]
with independent and identically distributed $\mathbb{R}^N$-valued Gaussian random variables
\[
X_{k+1} = P_N\int_{hk}^{h(k+1)} e^{(h(k+1)-s)A} dW(s)
\sim \mathcal{N}\Big(0,P_NQ\int_0^h e^{2sA}ds\;P_N\Big)\;.
\]

Given these values $Z_k$ for $Z(hk)$, the numerical method provides a realization of the approximation
\[
\{ u_k \}_{k=0,\ldots,M}  \subset H_N\;,
\]
which is defined recursively as $u_0 = P_N u_\star$ and
\[
u_n = e^{Ah} P_N u_{n-1} + \int_{0}^{h} e^{A(h-s)}ds   F_N(u_{n-1})
+ X_{n}\;.
\]
We can also write this explicitly as
\[
u_n = e^{Anh}P_N u_0+\sum_{k=1}^{n}\int_{(k-1)h}^{kh}e^{A(kh-s)}ds   F_N(u_{k-1})
+ Z_{n}\;.
\]

Moreover, we define the approximation $\varphi:[0,T] \to H_N$ by the linear interpolation of the points $\varphi(hk)=u_k$.

\subsection{Result}

The aim of this paper is to bound the conditional mean-square error given the numerical data, i.e, we want to obtain:
\[
\mathbb{E} [\|u-\varphi\|^2\ | \ \{ X_k \}_{k\in\N} ] =
\mathbb{E} [\|u-\varphi\|^2\ | \ \{ Z_k \}_{k\in\N} ] \text{ is small}\;.
\]
Therefore we do not want to estimate the error in an asymptotic result, but
give a bound can be explicitly calculated for the given realization of the approximation.

In Theorem \ref{thm:main} we present the main analytic result for this statement.
The term ''small'' depends on one hand on the the numerical data $Z$ and $X$, and we will evaluate this part  only numerically.
Thus we can only say that it is small, after we computed it.
On the other hand, we have infinite-dimensional parts and random fluctuations between discretization points,
which we have to bound analytically, as there is no data available.

The general philosophy is to evaluate as much as possible of the error bounds using the numerical data, and only rely
on analytic estimates if no numerical evaluation is possible.
As usual we consider in Section \ref{sec:Res} first the residual defined by:

\begin{definition}
\label{def:res}
For the approximation
$\varphi : [0,T] \to P_N H$ defined in Section \ref{sec:defphi} and
 $t \in [0,T]$ we define the residual
\begin{equation}
 \label{e:defRes}
\mathrm{Res} (t)= \varphi(t) - e^{At}\varphi(0) - \int_{0}^{t}e^{A(t-s)}F(\varphi(s))ds - Z(t).
\end{equation}
\end{definition}

We identify in the residual all parts depending on the numerical data, which we do not estimate at all, but evaluate explicitly
using the numerical data.

At first for the discretization points $t_n=nh$ we have
 \[
 \text{Res}(nh) = - \sum_{k=1}^{n}\int_{(k-1)h}^{kh}e^{A(nh-s)} (F(\varphi(s))-P_NF(u_{k-1}))ds -  Q_N Z(nh)\;.
\]
In  Lemma \ref{lem:Resnh} we estimate the residual at the discretization points $nh$.
As $\varphi((k-1)h+\tau) =  u_{k-1} + \frac{\tau}{h} (u_{k}-u_{k-1}) $ for $\tau\in(0,h)$
we can expand the cubic nonlinearity $F$ and evaluate all the integrals above explicitly.
Only the infinite $Q_N Z$ has to be estimated.

In order to bound the residual at intermediate times we first rewrite it in Lemma \ref {lem:resint}
and we present the main bounds on the residual in Theorem \ref{thm:mainres}.

A crucial term for the Theorem bounding the residual is 
the OU-bridge process that gives bounds on the stochastic convolution between discretization points.
The following Section \ref{sec:SB} provides the stochastic bounds on the infinite dimensional remainder of the OU-process and the OU-bridge process.

In view of the approximation result of Section \ref{sec:App} which is done by a standard a-priori estimate in $L^2$-spaces,
we need the bounds in $L^4$ on the residual, as we rely on the cubic nonlinearity. Moreover, the residual also contains a cubic,
so we need $L^{12}$-bounds of the data. 

As we want to obtain explicitly computable bounds for the Allen-Cahn equation, we have to rely on the special structure
of the equations. Nevertheless the general approach  (especially for the residual) 
can easily be adapted to other equations, and in Section \ref{sec:out}
we give a few comments on possible generalizations.


\section{Stochastic bounds}
\label{sec:SB}

Here we present analytic results for stochastic terms we cannot evaluate explicitly using numerical data.
One is the infinite-dimensional remainder of the stochastic convolution at discretization times.
The other one arises from fluctuations in between discretization times, where we need to analyze an Ornstein-Uhlenbeck bridge-process.

\subsection{OU-process}

For the stochastic term $Q_N Z(nh)$ we cannot use any numerical data to evaluate it.
Moreover it is infinite dimensional. The main result here is Lemma \ref{lemBouQ} below.
First we  need  estimates of a Gaussian in the  $L^4$-norm using the expansion in Fourier-series.
We use the $L^4$-norm, as this is the norm needed in the $L^2$-approximation result.
It should be straightforward to extend this to general $L^p$-spaces or even uniform bounds in space.

\begin{lemma}
Let $\mathcal{Z}\sim\mathcal{N}(0,\mathcal{Q})$ be a Gaussian 
with a covariance-operator $\mathcal{Q}$ on $H$ such that $\text{tr}(\mathcal{Q})<\infty$. 
Denote the eigenvalues and eigenfunctions by
$\mathcal{Q}e_k=a_k^2 e_k$
and suppose that for all $x$ we have $\sum_{k\in\mathbb{N}} a_k^2 e_k^2(x)<\infty$, then
\[
\mathbb{E}\|\mathcal{Z}\|_{L^4}^4 = 3 \sum_{k,\ell} a_k^2 a_\ell^2 \int_0^\pi e_k^2(x)e_\ell^2(x) dx \;.
\]
  \end{lemma}

  \begin{proof}
   By the properties of the covariance operator, we can expand
   \[
   \mathcal{Z} = \sum_{{k\in\mathbb{N}}} a_k n_k e_k  
   \]
   for a family $\{n_k\}_{k\in\mathbb{N}}$ of independent standard real-valued Gaussians.
   By assumption, we obtain that for all $x$ the real-valued random variable
   \[\mathcal{Z}(x)=\sum_{{k\in\mathbb{N}}} a_k n_k e_k(x)  \sim \mathcal{N}\Big(0, \sum_{k\in\mathbb{N}} a_k^2 e_k^2(x)\Big)
   \text{ in }\mathbb{R} 
   \]
   and the sequences above converges in $\mathbb{R}$ in mean square with respect to the probability measure.

   Thus we can use the fact that all moments of a centered real-valued Gaussian can be computed using only the second moment,
   to obtain by Tonelli's theorem
   \[
   \mathbb{E}\|\mathcal{Z}\|_{L^4}^4 = \int  \mathbb{E}|\mathcal{Z}(x)|^4 dx = 3\int ( \mathbb{E}|\mathcal{Z}(x)|^2)^2 dx  =  3\int (  \sum_{k} a_k^2 e_k^2(x))^2 dx
   \]
   which implies the claim.
  \end{proof}

Recall that the Fourier-basis $\{e_k\}_{k\in\mathbb{N}}$ with respect to Dirichlet boundary conditions on $[0,\pi]$ is defined by
\[
e_{k}(x)=\sqrt2 \sin(kx) / \sqrt{\pi}.
\]
Here we have for $k=1,2,\ldots$ where $f_k(x)= \sqrt2 \cos(kx) / \sqrt{\pi}$
 \[
 e_{k}^2(x)= \frac2{\pi} \sin^2(kx) = \frac1{\pi} -\frac1{\sqrt{2\pi}}f_{2k}(x) \;.
\]
Thus we have for $k,\ell>0$
\[
 \int_0^{\pi} e_{k}^2(x)e_{\ell}^2(x) dx =   \frac1{\pi} +  \frac1{2\pi} \delta_{k,\ell}\;.
\]
Now we can verify
\[
\begin{split}
 \int_0^{\pi}   \sum_{k,\ell} a_k^2a_\ell^2 e_k^2(x) e^2_\ell(x)  dx
 & = \frac1{\pi} \Big( \sum_{k}a_k^2\Big)^2 + \frac1{2\pi} \sum_{k}a_k^4 \\
  & \leq \frac3{2\pi} \Big( \sum_{k}a_k^2\Big)^2 \;.
\end{split}
\]
This yields the following lemma:
\begin{lemma}
\label{lem:OUL4}
Let $\mathcal{Z}\sim\mathcal{N}(0,\mathcal{Q})$ be a Gaussian on $H$ with a covariance-operator $\mathcal{Q}$ such that $\text{tr}(\mathcal{Q})<\infty$.
Let $\mathcal{Q}$ be diagonal w.r.t.\ the Fourier-basis $e_k$, then
\[
\mathbb{E}\|\mathcal{Z}\|_{L^4}^4 \leq  \frac3{2\pi} (\text{tr}(\mathcal{Q}))^2 = \frac3{2\pi} (\mathbb{E}\|\mathcal{Z}\|_{L^2}^2)^2\;.
\]
  \end{lemma}
 We finally obtain for our infinite dimensional OU-process $Z$ from the numerical approximation:
\begin{lemma}\label{lemBouQ}
For all $N\in\mathbb{N}$
	the sequence $\{Q_N Z(kh)\}_{k=1,\ldots, M}$ is independent of $(Z_k)_{k\in\mathbb{N}}$
	and bounded in the $L^4$-norm  by
	\[
	 \sup_{t \geq 0} \mathbb{E} \|Q_NZ(t)\|_{L^4}^4 \leq \frac{3}{8\pi} \Big(\sum_{k>N} \alpha_k^2k^{-2}\Big)^2.
	\]
\end{lemma}

A stronger result is proven in \cite{DBAJ:Galerkin}, 
where  we even could take the supremum in time
over bounded intervals inside the expectation and thus use $L^\infty$- instead of $L^4$-norms.
But the constant in \cite{DBAJ:Galerkin} is not calculated explicitly.

The bound of Lemma \ref{lemBouQ} is still not numerically computable, but given a bounded sequence $\alpha_k$,
it is  usually straightforward to evaluate (or bound)  the series explicitly. See Remark \ref{rem:mainres}.

\begin{proof}
We start by using Lemma \ref{lem:OUL4}, as $Q_N Z(t) \sim \mathcal{N} (0, Q_N Q \int_0^t e^{2sA}ds)$ in $H$, 
to obtain
\[
 \begin{split}
   \mathbb{E} \|Q_NZ(t)\|_{L^4}^4
   &\leq  \frac{3}{2\pi} (  \mathbb{E} \|Q_NZ(t)\|_{L^2}^2 )^2 \\
   &= \frac{3}{2\pi} \Big(  \sum_{k>N} \alpha_k^2 \int_0^t e^{-2k^2s}ds \Big)^2\;.
 \end{split}
\]
This easily implies the claim.
\end{proof}

\subsection{OU-bridge}

In order to treat random fluctuations between discretization points, we
define for $\tau\in(0,h)$
 \begin{equation}
\label{e:defcZ}
\mathcal{Z}_n(\tau) =\int_{nh}^{nh+\tau}e^{A(nh+\tau-s)}dW_{s}\;.
 \end{equation}
First the processes $ \{\mathcal{Z}_n\}_{n\in\mathbb{N}}$ are independent and identically distributed.
Denote also the high modes  $\mathcal{Z}^{(h)}_n = (I-P_N)\mathcal{Z}_n$ and the low modes  $\mathcal{Z}^{(l)}_n = P_N\mathcal{Z}_n$,
which are by definition mutually independent.

Moreover, it is easy to see that
$\{\mathcal{Z}_n(\tau)\}_{\tau\in[0,h]}$
depends on  $ \{Z_k\}_{k\in\mathbb{N}}$
only via $\mathcal{Z}^{(l)}_n(h) = Z_{n+1}-e^{hA}Z_n$.
Thus, recalling that $ \{Z_k\}_{k\in\mathbb{N}}$ is a Markov-process we obtain
\[
\mathbb{E}[\int_0^h\|  \mathcal{Z}_n(\tau)\|_{L^4}^4 d\tau  | (Z_k)_{k \in \mathbb{N}}]
=  \mathbb{E}[\int_0^h\|  \mathcal{Z}_n(\tau)\|_{L^4}^4 d\tau | \mathcal{Z}^{(l)}_{n}(h)] \;.
\]
This yields the following Lemma:
   \begin{lemma}
   \label{lem:defN}
\[
\mathbb{E}[\int_0^h\|  \mathcal{Z}_n(\tau)\|_{L^4}^4 d\tau | (Z_k)_{k \in \mathbb{N}} ] = \mathcal{N}(Z_{n+1} - e^{hA}Z_n)
\]
where
\[
 \mathcal{N}(z) = \mathbb{E}[\int_0^h\|  \mathcal{Z}_0(\tau)\|_{L^4}^4 d\tau | \mathcal{Z}^{(l)}_0(h)=z]
 \]
 \end{lemma}

Note that this splits into the infinite dimensional remainder and an OU-bridge process on the low modes.
For the OU-bridge process by a result of \cite{Orn:Gold} we know explicitly the law:

 \begin{lemma}
 \label{lem:OUB}
For $t\in[0,h]$
the law of $\mathcal{Z}^{(l)}_0(t)$
given $ \mathcal{Z}^{(l)}_0(h)=z $ with $z\in P_N H$
is a Gaussian with mean $\lambda(t,z)$ and covariance $\tilde{\mathcal{Q}}_t$
with
\[
\lambda(t,z) = 2 P_N \int_0^t e^{2As}ds \cdot e^{A(h-t)}[I-e^{2hA}]^{-1} A z
\]
and
\[
\tilde{\mathcal{Q}}_t=  P_N\frac{Q}{2A} \frac{1-e^{2tA}}{1-e^{2hA}}  \left(I -  e^{2A(h-t)}\right) \;.
\]
\end{lemma}
\begin{proof}
We follow the result of \cite{Orn:Gold}, but our setting is much simpler.
First all operators involved are diagonal and thus symmetric. Furthermore, they all commute.
We can also treat degenerate noise by restricting the results of \cite{Orn:Gold} to the
Hilbert-space given by the range of $Q$, which is in general only a subset of $P_N H$.
But then both $Q$ and $A$ are invertible on that space.

First,
\[ \text{Law}[\mathcal{Z}^{(l)}_0(t)| \mathcal{Z}_0^{(l)}(h)=z] =  \mathcal{N}(\lambda(t,z),\tilde{\mathcal{Q}}_t)
\]
where by
\cite[Prop. 2.11]{Orn:Gold}
\[
\tilde{\mathcal{Q}}_t = P_N Q_t(I-V_t^2)
\]
and by \cite[Prop. 2.13]{Orn:Gold}
\[
\lambda(t,z) = P_N K_tQ_h^{-1/2}z
\]
with the following operators,
\[ Q_t= Q \int_{0}^{t}e^{2As}ds , \quad
		 V_{t}=Q_{h}^{-\frac 12}e^{A(h-t)}Q_{t}^{\frac 12} , \quad
		 K_t=Q_t^{\frac12}V_{t} = Q_{h}^{-\frac 12}e^{A(h-t)}Q_{t}\;.
\]
This implies
\[
\lambda(t,z) = P_NQ_{h}^{-1}e^{A(h-t)}Q_{t}z = P_N\int_{0}^{t}e^{2As}ds \cdot e^{A(h-t)}\Big[\int_0^h  e^{2As}ds \Big]^{-1} z\;.
\]
and
\[
\begin{split}
 \tilde{\mathcal{Q}}_t
 = &  P_NQ_t(I- Q_{h}^{-1}e^{2A(h-t)}Q_{t}) \\
 =&  P_N Q_t Q_h^{-1}(Q_h - e^{2A(h-t)}Q_{t}) \\
  =&  P_N Q \frac{1-e^{2tA}}{1-e^{2hA}}  \frac12 A^{-1} \left(I - e^{2hA}-  e^{2A(h-t)}(1-e^{2tA})\right) \\
  =& P_N  \frac{Q}{2A} \frac{1-e^{2tA}}{1-e^{2hA}}  \left(I -  e^{2A(h-t)}\right)\;.
\end{split}
\]
 \end{proof}

   The following  bound is surely not optimal, but a slight simplification of the exact bound.

 \begin{lemma}
 \label{lem:boundN} We bound for $z\in P_N H$
\[\mathcal{N}(z) \le  h \cdot \mathcal{S}_h(z)\]
where  we define
\[
\mathcal{S}_h(z) = \Big[\frac{2 h }{\sqrt[4]{5}}  \|[I-e^{2hA}]^{-1} A z\|_{L^4}
+
\Big(\frac{3}{2\pi}\Big)^{1/4} \left( h^{-1/2} \Sigma_N(h) + \sum_{k=N+1}^\infty  \frac{\alpha_k^2 }{2k^2}   \right)^{1/2}   \Big]^4
\]
where
\[ \Sigma_N(h)
 = \sum_{k=1}^N \frac{\alpha_k^2 }{\sqrt8 \cdot k^3} \frac{(8hk^2e^{-2hk^2}+(2hk^2+3)e^{-4hk^2}+2hk^2-3 )^{1/2}}{1-e^{-2hk^2}} \;.
\]

 \end{lemma}

  We can explicitly calculate an upper bound for $\mathcal{S}_h$ by first using numerical data for $z$,
  and then for the infinite
  series we can use
  \[
   \sum_{k>N}  \frac{\alpha_k^2 }{k^{2}}
   \leq   \frac{\sup_{k>N}\alpha_k^2}{N} \;.
  \]

  \begin{proof}
   First using Lemmas \ref{lem:OUB} and \ref{lem:OUL4} and taking into account
   the infinite-dimensional remainder of the OU-process, that is independent of the OU-bridge, we obtain
\[ \text{Law}[\mathcal{Z}_0(t)| \mathcal{Z}_0^{(l)}(h)=z] =  \mathcal{N}(\lambda(t,z),\tilde{\mathcal{Q}}_t)
\]
on $H$
  with covariance operator $\hat{\mathcal{Q}}_t$ being diagonal in Fourier space with
\[P_N\hat{\mathcal{Q}}_t = \tilde{\mathcal{Q}}_t
\quad \text{and} \quad
(I-P_N) \hat{\mathcal{Q}}_t  = Q \int_0^t e^{2sA}ds
\]
where
   \[
    \mathcal{N}(z)^{1/4}
    \leq
    \Big(\int_0^h\|\lambda(s,z)\|_{L^4}^4 ds \Big)^{1/4}
    +  \Big( \frac{3}{2\pi} \int_0^h \text{trace}\left(  \tilde{\mathcal{Q}}_s \right)^2 ds\Big)^{1/4}
   \]
   Now
   \[
   \begin{split}
      \text{trace}(  \hat{\mathcal{Q}}_s ) &= \text{trace}(  \tilde{\mathcal{Q}}_s ) + \text{trace}(  Q \int_0^t e^{2\eta A}d\eta )
      \\&
   =  \sum_{k=1}^N \frac{\alpha_k^2 }{2k^2} \frac{1-e^{-2sk^2}}{1-e^{-2hk^2}}  \left(1 -  e^{-2k^2(h-s)}\right)
   + \sum_{k=N+1}^\infty \frac{\alpha_k^2 }{2k^2}   \left(1 -  e^{-2k^2s}\right)
   \end{split}
   \]
   and thus
   \[
    \begin{split}
    \int_0^h\text{trace}(  \hat{\mathcal{Q}}_s )^2  ds
    &=  \| \text{trace}(\hat{\mathcal{Q}}_s)\|_{L^2(0,h)}^2 \\
    & \leq
    \Big(  \| \text{trace}(\tilde{\mathcal{Q}}_s)\|_{L^2(0,h)} +\| \text{trace}(Q \int_0^\cdot  e^{2\eta A}d\eta  )\|_{L^2(0,h)}  \Big)^2
   \end{split}
   \]
   where we bound
   \[
   \begin{split}
    &\| \text{trace}(\tilde{\mathcal{Q}}_s)\|_{L^2(0,h)} \\
    \leq &
    \sum_{k=1}^N \frac{\alpha_k^2 }{2k^2} \frac1{1-e^{-2hk^2}}      \| (1-e^{-2sk^2})(1 -  e^{-2k^2(h-s)})\|_{L^2(0,h)} \\
    \leq&
     \sum_{k\in\mathbb{N}} \frac{\alpha_k^2 }{2k^2} \frac1{1-e^{-2hk^2}} \frac1{\sqrt{2k^2}} \Big(8hk^2e^{-2hk^2}+(2hk^2+3)e^{-4hk^2}+2hk^2-3 \Big)^{1/2}  \\
    \leq &
   \sum_{k\in\mathbb{N}} \frac{\alpha_k^2 }{\sqrt8 \cdot k^3} \frac{(8hk^2e^{-2hk^2}+(2hk^2+3)e^{-4hk^2}+2hk^2-3 )^{1/2}}{1-e^{-2hk^2}}
   =\Sigma_N(h) \;.
 \end{split}
   \]
  Moreover,
  \[
  \| \text{trace}(Q \int_0^\cdot  e^{2\eta A}d\eta  )\|_{L^2(0,h)}
   \leq  \sum_{k=N+1}^\infty  \frac{\alpha_k^2 }{2k^2}\| 1-e^{-2k^2t}\|_{L^2(0,h)}
   \leq h  \sum_{k=N+1}^\infty  \frac{\alpha_k^2 }{2k^2}\;.
  \]
 Thus
 \[
   \int_0^h\text{trace}(  \hat{\mathcal{Q}}_s )^2  ds
    \leq
    \left(\Sigma_N(h) +  h \sum_{k=N+1}^\infty  \frac{\alpha_k^2 }{2k^2}   \right)^2
   \]
   For the mean value, we obtain from Lemma \ref{lem:OUB} using \eqref{e:RT}
   \[
   \|\lambda(t,z)\|_{L^4}   \leq  2t \|[I-e^{2hA}]^{-1} A z\|_{L^4} \;.
   \]
   Thus
   \[
   \int_0^h \|\lambda(t,z)\|_{L^4}^4 dt   \leq  \frac{2^4}{5} h^5 \|[I-e^{2hA}]^{-1} A z\|^4_{L^4} \;.
   \]
  \end{proof}


\section{Residual estimates}
\label{sec:Res}


This section is devoted to bounds on the residual, which measures the quality of an arbitrary numerical approximation.
First we consider the discretization points $t_n=nh$,
and later we focus on the points $nh+\tau$, $\tau \in (0,h)$, which are in between.
Recall that for the approximation
$\varphi : [0,T] \to P_N H$ defined in Section \ref{sec:defphi} and
 $t \in [0,T]$ we defined the residual in Definition \ref{def:res}.

\subsection{At discretization points}

In the following Lemma we identify all terms  in the residual
at the discretization points $nh$ that can be calculated explicitly using the numerical data. We define them as $\text{Res}^{\text{dat}}$.
\begin{lemma}\label{lem:Resnh}
	The residual $\{\text{Res}(kh)\ :\ k=0,\ldots, M \}$ defined in \eqref{e:defRes} at discrete times is given as
	\[
	\text{Res}(kh) = \text{Res}^{\text{dat}}_k  +  \text{Res}^{\text{stoch}}_k
	\]
	where
	\[ \text{Res}^{\text{dat}}_k = I_2(k)+I_3(k)
	  \]
	  given by the recursive scheme $I_j(0)=0$,
	  \[
	  I_2(n)=e^{hA} I_2(n-1) + h  Q_N \int_{0}^{1} e^{sAh } \Big[ u_{n}^3
- s   3(u_{n})^2 d_n + s^2  3u_{n}(d_n)^2
-  s^3 (d_n)^3\Big] ds
	  \]
	  and
	\[
	  I_3(n)=  e^{hA} I_3(n-1) + h P_N \int_{0}^{1} e^{sAh}  \Big[(u_n^3- u_{n-1}^3)
 -3 s d_nu_{n}^2
 +3 s^2  d_n^2 u_{n}
 - s^3 d_n^3 \Big]ds.
\]
	Moreover,
	\[\text{Res}^{\text{stoch}}_n = - Q_N Z(nh)\]
        is random and independent of the numerical data.
\end{lemma}

The value of $\text{Res}^{\text{dat}}$ just depends on the numerical data $u_k$ and $d_k=u_k-u_{k-1}$.
Note that the cubic terms depending on $u_k$ and $d_k$ are all in $H_{3N}$
and thus computable. The integrals are all over diagonal matrices and can be calculated explicitly in the numerical evaluation.

\begin{proof}
We have
\begin{equation}\label{Resnh}
\begin{split}
 \text{Res}(nh)
 = &  -
  \sum_{k=1}^{n}\int_{(k-1)h}^{kh} e^{A(nh-s)} Q_N F(\varphi(s)) ds
 \\
 & -\sum_{k=1}^{n}\int_{(k-1)h}^{kh}  e^{A(nh-s)}  P_N [(F(\varphi(s))-F(u_{k-1})]\; ds
 \\& -  Q_N Z(nh)
 \\ =:  & I_2(n) + I_3(n) -  Q_N Z(nh)\;.
 \end{split}
\end{equation}
For the two integrals we use for $s\in[(k-1)h,kh]$
that  $\varphi(s)= u_{k-1}+\frac{\tau}{h}d_k$, where $s=(k-1)h+\tau$ with $\tau\in(0,h)$ and $d_k=u_k-u_{k-1}$,
which depends only on the numerical data.

For the first integral $I_2$ on the right hand side we note that although it looks infinite dimensional,
due to the cubic nonlinearity it is finite dimensional and we can calculate it explicitly:
\[
\begin{split}
I_2(n)  & := -\sum_{k=1}^{n} \int_{0}^{h} e^{A(h(n-k+1)- s)} Q_N F(u_{k-1}+\frac{s}{h}d_k) ds \\
& =  -h \sum_{k=1}^{n} \int_{0}^{1} e^{Ah (n-k+ s)}  Q_N F(u_{k}-s d_k)ds \\
&=  e^{hA} I_2(n-1)-h\int_{0}^{1} e^{Ah s}  Q_N F(u_{k}-s d_k)ds\\
&=  e^{hA} I_2(n-1) + h  Q_N \Big[\int_{0}^{1} e^{As h} ds  \cdot  u_{n}^3
- \int_{0}^{1} e^{Ahs}s ds \cdot  3(u_{n})^2 d_n \Big] \\
& \qquad +h Q_N\Big[ \int_{0}^{1} e^{Ahs} s^2 ds  \cdot  3u_{n}(d_n)^2
- \int_{0}^{1} e^{Ahs} s^3 ds  \cdot (d_n)^3\Big]
 \end{split}
\]

For the next integral $I_3$ in (\ref{Resnh}) we can proceed similarly as for $I_2$
 \begin{equation*}
\begin{split}
I_3(n)
& = \sum_{k=1}^{n} \int_{(k-1)h}^{kh}  e^{A(nh-s)} P_N \left[ \varphi(s)^3 - (u_{k-1})^3 \right] ds
\\=
&  \sum_{k=1}^{n}\int_{(k-1)h}^{kh} e^{A(nh-s)} P_N \left[ (u_{k-1}+\frac{s-(k-1)h}{h}d_k))^3 - (u_{k-1})^3 \right] ds
  \\
= & h \sum_{k=1}^{n}\int_{0}^{1}  e^{Ah(n-k+1-s)} P_N \left[(u_{k-1}+s d_k)^3- (u_{k-1})^3\right] ds
 \\= &   h \sum_{k=1}^{n}\int_{0}^{1} e^{Ah(n-k+s)} P_N \left[(u_{k-1}+(1-s)d_k)^3- (u_{k-1})^3\right] ds\\
 =& h \sum_{k=1}^{n}\int_{0}^{1} e^{Ah(n-k+s)} P_N \left[(u_{k}-sd_k)^3- (u_{k-1})^3\right] ds\\
 = & e^{hA} I_3(n-1) + h \int_{0}^{1} e^{sAh} P_N \left[(u_{n}-sd_n)^3- (u_{n-1})^3\right] ds \;.
 \end{split}
\end{equation*}
By expanding the cubic we have
\[
 \begin{split}
 I_3(n)& =  e^{hA} I_3(n-1) + h \int_{0}^{1} e^{sAh}P_N ds \cdot (u_n^3- u_{n-1}^3)
 -3h \int_{0}^{1} e^{sAh}sds P_N \cdot  d_nu_{n}^2\\&
 +3h\int_{0}^{1} e^{sAh}P_N s^2 ds \cdot d_n^2 u_{n}
 -h \int_{0}^{1} e^{sAh}P_N s^3 ds \cdot d_n^3.
 \end{split}
\]
\end{proof}


 \subsection{Between discretization points}

 For the residual at times between the numerical grid points we have for $\tau\in(0,h)$
   \begin{equation}
\begin{split}
- &\text{Res} (nh+\tau) \\
= &
e^{A(nh+\tau)}\varphi(0)+\int_{0}^{nh+\tau}e^{A(nh+\tau-s)}F(\varphi(s))ds
+Z(hn+\tau)-\varphi(nh+\tau)
\\= & e^{A(nh+\tau)}\varphi(0)
+\int_{0}^{nh}e^{A(nh+\tau-s)}F(\varphi(s))ds+\int_{nh}^{nh+\tau}e^{A(nh+\tau-s)}F(\varphi(s))ds
\\&+e^{A\tau}Z_n+\int_{nh}^{nh+\tau}e^{A(nh+\tau-s)}dW_{s}-\varphi(nh+\tau)
\\ =&e^{A\tau}\Big[e^{nhA}\varphi(0)+\int_{0}^{nh}e^{A(nh-s)}F(\varphi(s))ds+\int_{nh}^{nh+\tau}e^{A(nh-s)}F(\varphi(s))ds+Z_{n}\Big]
\\&+\int_{nh}^{nh+\tau}e^{A(nh+\tau-s)}dW_{s}-\varphi(nh+\tau)
\\
=&  e^{A\tau}\Big[\text{Res}(nh)+u_{n}+\int_{nh}^{nh+\tau}e^{A(nh-s)}F(\varphi(s))ds\Big]
\\&
+\int_{nh}^{nh+\tau}e^{A(nh+\tau-s)}dW_{s}-\varphi(nh+\tau).
 \end{split}
\end{equation}
 Therefore by the fact that by linear interpolation $\varphi(nh+\tau)=u_{nh}+\frac{\tau}{h}d_{n+1}$,
 where $d_{n+1}=u_{n+1}-u_n$ we get
\begin{equation}\label{Res}
\begin{split}
\text{Res}(nh+\tau) 
= & e^{A\tau}\text{Res}(nh) +(e^{A\tau}-I)u_n
-\frac{\tau}{h}d_{n+1}
\\&+
\int_{nh}^{nh+\tau}e^{A(nh+\tau-s)}F(\varphi(s))ds
+\int_{nh}^{nh+\tau}e^{A(nh+\tau-s)}dW_{s}.
 \end{split}
\end{equation}
Now at some point we cannot evaluate and need to estimate, as due to the $\tau\in(0,h)$, 
we cannot evaluate the terms numerically explicit.

 For the first integral term in the right hand side of (\ref{Res}) we have
   \begin{equation}
   \label{e:defI}
I(\tau)=\int_{nh}^{nh+\tau}e^{A(nh+\tau-s)}F(\varphi(s))ds=
-h\int_0^{\frac{\tau}{h}}e^{A(\tau-sh)}(u_n+sd_{n+1})^3ds\;.
\end{equation}

In order to bound $I(\tau)$, we use that the semigroup $e^{tA}$  generated by the Dirichlet-Laplacian $A$ is a contraction semigroup on any $L^p(0,\pi)$.
See \eqref{e:RT}. Thus we obtain

\begin{equation}
\label{e:bouI}
\begin{split}
\|I(\tau)\|_{L^4}
 \leq  & h \int_0^{\frac{\tau}{h}} \|(u_n+sd_{n+1})^3\|_{L^4}ds \\
 =  & h \int_0^{\frac{\tau}{h}} \|(u_n+sd_{n+1})\|_{L^{12}}^3ds \\
 \leq  & h \int_0^{\frac{\tau}{h}} \left(\|u_n\|_{L^{12}} +s \|d_{n+1}\|_{L^{12}} \right)^3 ds \\
 \leq  &  \tau \|u_n\|_{L^{12}}^3
 + \frac32\frac{\tau^2}{h}  \|u_n\|_{L^{12}}^2\|d_{n+1}\|_{L^{12}}
 \\&
 + \frac{\tau^3}{h^2}  \|u_n\|_{L^{12}}\|d_{n+1}\|_{L^{12}}^2
 + \frac14\frac{\tau^4}{h^3} \|d_{n+1}\|_{L^{12}}^3\;.
 \end{split}
\end{equation}

This still contains powers of $\tau$, but as we are going to integrate this over $\tau$, we keep them and estimate later.
Let us summarize the result starting from \eqref{Res}.

\begin{lemma}
\label{lem:resint}
For $n\in\{0,\ldots,M-1\}$ and $\tau\in[0,h]$ we have
\begin{equation}
\text{Res}(nh+\tau)
=  e^{A\tau}\text{Res}(nh) +(e^{A\tau}-I)u_n
-\frac{\tau}{h}d_{n+1}+
I(\tau)
+\mathcal{Z}_n(\tau),
\end{equation}
with $I$ defined in \eqref{e:defI} and bounded in \eqref{e:bouI}
and $\mathcal{Z}_n$ was defined in \eqref{e:defcZ}.
\end{lemma}


\subsection{Bounding the residual}

Now we are going to bound the residual for intermediate times. Let us fix $n\in\{0,\ldots,M-1\}$ and $\tau\in[0,h]$.
In view of Lemma \ref{lem:resint}, we first bound
\[
\|(e^{A\tau}-I)u_n\|_{L^4} \leq  \| \int_0^\tau A e^{As}ds u_n\|_{L^4} \leq \tau \|Au_n\|_{L^4}
\]
to obtain
 \begin{equation}\label{Resnhtau}
 \|\text{Res}(nh+\tau)\|_{L^4}
  \leq  \|\text{Res} (nh)\|_{L^4} + \tau \|Au_n\|_{L^4} +  \frac{\tau}h\|d_{n+1}\|_{L^4}
  + \|I(\tau)\|_{L^4}+\|Z_n(\tau)\|_{L^4}.
 \end{equation}
In order to bound the residual in a conditional $L^4$-moment, we define
\[
n(s)=n \iff s\in[nh,h(n+1)) \quad \text{and}\quad \tau(s)=s-n(s)h \in[0,h)\;.
\]
We fix $t\in[0,T]$ and obtain from  (\ref{Resnhtau}) by triangle inequality
\begin{equation}
\begin{split}
 \Big(\mathbb{E}&[\int_{0}^{t} \|\text{Res}(\tau)\|_{L_4}^4 d\tau|(Z_k)_{k \in \mathbb{N}}] \Big)^{1/4}
  \\&=
 \Big( \mathbb{E}[\int _{0}^{t} \|\text{Res}(n(s)h+\tau(s))\|_{L_4}^4 ds|(Z_k)_{k \in \mathbb{N}}] \Big)^{1/4}
 \\&
  \leq
  \Big(\mathbb{E}[\int_{0}^{t}   \|\text{Res} (n(s)h)\|^4_{L^4} ds |(Z_k)_{k \in \mathbb{N}}] \Big)^{1/4}
  +  \Big(\int_{0}^{t}\tau(s)^4 \|Au_{n(s)}\|_{L^4}^4 ds  \Big)^{1/4}
  \\&
   + \Big(\int_{0}^{t}\frac{\tau(s)^4}{h^4}\|d_{n(s)+1}\|^4_{L^4}ds \Big)^{1/4}
  +  \Big(\int_{0}^{t}\|I(\tau(s))\|^4_{L^4}ds \Big)^{1/4}
  \\&
  + \Big(\mathbb{E}[\int_{0}^{t}\|Z_{n(s)}(\tau(s))\|^4_{L^4} ds |(Z_k)_{k \in \mathbb{N}}\Big)^{1/4}
  .
 \end{split}
\end{equation}
We now bound all the terms above separately.
Let $m(t)$ be the largest integer, such that $m(t)h \leq t$.
From Lemma \ref{lem:Resnh} again by triangle inequality, we get
\begin{equation*}
\begin{split}
\Big(\mathbb{E}&[\int_{0}^{t} \|\text{Res}(n(s)h)\|_{L_4}^4 ds |(Z_k)_{k \in \mathbb{N}}] \Big)^{1/4}
\\&  \leq  \Big(\int_{0}^{t} \|\text{Res}^{\text{dat}}_{n(s)}\|_{L_4}^4 ds  \Big)^{1/4}
+ \Big( \mathbb{E}\int_{0}^{t} \|\text{Res}^{\text{stoch}}_{n(s)}\|_{L_4}^4 ds  \Big)^{1/4}
\\&
\leq  \Big(h\sum_{n=1}^{m(t)} \|\text{Res}^{\text{dat}}_n\|_{L_4}^4  \Big)^{1/4}
+ \Big(h \sum_{n=1}^{m(t)} \mathbb{E}\|\text{Res}^{\text{stoch}}_n\|_{L_4}^4  \Big)^{1/4}
\\&
\leq
 \Big(h\sum_{n=1}^{m(t)} \|\text{Res}^{\text{dat}}_n\|_{L_4}^4  \Big)^{1/4}
+    \Big(\frac{3hm(t)}{8\pi}\Big)^{1/4} \Big( \sum_{k>N} \frac{\alpha_k^2}{k^2}  \Big)^{1/2}
\end{split}
\end{equation*}
where we used that $\text{Res}(0)=0$, so that all sums start at $n=1$.

For the next term
\[
\begin{split}
 \Big(\int_{0}^{t}\tau(s)^4 \|Au_{n(s)}\|_{L^4}^4 ds  \Big)^{1/4}
 \leq &
 \Big( \sum_{n=0}^{m(t)}\int_{nh}^{(n+1)h}\tau(s)^4 ds \|Au_{n}\|_{L^4}^4 ds  \Big)^{1/4}   \\
  \leq &
 \Big( \sum_{n=0}^{m(t)}\int_{0}^{h}s^4 ds \|Au_{n}\|_{L^4}^4 ds  \Big)^{1/4}   \\
 \leq &  \Big( \frac{h^5}{5}  \sum_{n=0}^{m(t)} \|Au_{n}\|_{L^4}^4 ds  \Big)^{1/4}
 \end{split}
\]
and similarly
\[\Big(\int_{0}^{t}\frac{\tau(s)^4}{h^4}\|d_{n(s)+1}\|^4_{L^4}ds \Big)^{1/4}
\leq \Big( \frac{h}{5}  \sum_{n=0}^{m(t)} \|d_{n+1}\|_{L^4}^4 ds  \Big)^{1/4} \;.
\]
For the integral-term $I$ by  \eqref{e:bouI}
\[ \begin{split}
\Big(\int_{0}^{t}\|I(\tau(s))\|^4_{L^4}ds \Big)^{1/4}
\leq \Big( h^2 \sum_{n=0}^{m(t)} \Big[
&
\frac12 \|u_n\|_{L^{12}}^3
 + \frac12 \|u_n\|_{L^{12}}^2\|d_{n+1}\|_{L^{12}}
 \\&
 + \frac14  \|u_n\|_{L^{12}}\|d_{n+1}\|_{L^{12}}^2
 + \frac1{20} \|d_{n+1}\|_{L^{12}}^3
 \Big]   \Big)^{1/4} \;.
    \end{split}
    \]
Finally for the Ornstein-Uhlenbeck bridge, we have from Lemma \ref{lem:boundN}
\[\Big(\mathbb{E}[\int_{0}^{t}\|Z_{n(s)}(\tau(s)) \|^4_{L^4} ds |(Z_k)_{k \in \mathbb{N} ] }\Big)^{1/4}
\leq  \Big( h \sum_{n=0}^{m(t)} \mathcal{S}_h(Z_{n+1}-e^{hA}Z_n) ) \Big)^{1/4}\;.
\]

Summarizing, we have the following bound:

 \begin{theorem}
 \label{thm:mainres}
 With numerical data $\varphi$ from Section \ref{sec:defphi} and the residual defined in \eqref{e:defRes}   we have
 for $t\in[mh,(m+1)h)$ and $m\in\{1,\ldots,M\}$ with $M=T/h$
 \[
 \mathbb{E}[\int_{0}^{t} \|\text{Res}(s)\|_{L_4}^4 ds|(Z_k)_{k \in \mathbb{N}}]
\leq \mathcal{K}_m^4
\]
with
\[
\begin{split}
\mathcal{K}_m =
&
\Big(h\sum_{n=1}^{m} \|\text{Res}^{\text{dat}}_n\|_{L_4}^4  \Big)^{1/4}
+    \Big(\frac{3mh}{8\pi}\Big)^{1/4} \Big( \sum_{k>N} \frac{\alpha_k^2}{k^2}  \Big)^{1/2}
\\&
+ \Big( \frac{h^5}{5}  \sum_{n=0}^{m} \|Au_{n}\|_{L^4}^4 ds  \Big)^{1/4}
 + \Big(h \sum_{n=0}^{m} \mathcal{S}_h(Z_{n+1}-e^{hA}Z_n) \Big)^{1/4}
\\&
+ \Big( h^2 \sum_{n=0}^{m} \Big[
\frac12 \|u_n\|_{L^{12}}^3
 + \frac12 \|u_n\|_{L^{12}}^2\|d_{n+1}\|_{L^{12}}
\\& \qquad\qquad\qquad
+ \frac14  \|u_n\|_{L^{12}}\|d_{n+1}\|_{L^{12}}^2
 + \frac1{20} \|d_{n+1}\|_{L^{12}}^3
 \Big]   \Big)^{1/4}\;.
\end{split}
\]
 \end{theorem}

\begin{remark}
\label{rem:mainres}
The quantity $\mathcal{K}_m$ is almost numerically computable using numerical data.
Moreover, we can update the sums in the numerical computation, so that we do not need to calculate them in every step.

The only term that is not yet computable is
the sum depending on the $\alpha_k$ for $k>N$, but here one can easily give an upper bound, once the $\alpha_k$ are given, by
 \[
 \sum_{k>N} \frac{\alpha_k^2}{k^2}  \leq  \sup_{k>N}\{\alpha_k^2\} \int_N^\infty k^{-2} dk = \frac{1}N \sup_{k>N}\{\alpha_k^2\}\;.
 \]
\end{remark}

Let us also remark that due to the way we did the estimate, we cannot take the number $N$ of Fourier-modes arbitrarily large.
Due to the regularity of the solution $u$, which is not in $H^2$, we cannot expect $\|Au_n\|$ to be bounded for $N\to\infty$.
Thus we always need to take $h$ sufficiently small to balance that effect.

We expect that it is possible to give a precise estimate for the asymptotic limit $h\to 0$ and $N\to \infty$ of $h\|Au_n\|$,
but here we intend to calculate this explicitly, in order to obtain a better bound without any estimate.


 \section{Approximating the error}
\label{sec:App}

In this section, the arguments crucially depend on the properties of the equation especially on the nonlinear stability.
The numerical data only comes into play via the residual.
We need to quantify the continuous dependence of solutions on additive perturbations given by the residual.
Recall the mild solution of (\ref{SPDE})
\[
u(t) =e^{At}u_\star+\int_{0}^{t}e^{A(t-s)}F(u(s))ds+ Z(t),
\]
and the definition of the residual
\[
 \varphi(t) = e^{At} \varphi(0) + \int_{0}^{t}e^{A(t-s)}F(\varphi(s))ds + Z(t) + \text{Res}(t)\;.
\]
Therefore by putting $d(t)=u(t)-\varphi(t)$ we have
\[
d(t)=u(t)-\varphi(t) = e^{tA}d(0)
+\int_{0}^{t}e^{A(t-s)}(F(u(s)) - F(\varphi(s)))ds-\text{Res}(t)
\]
with $d(0)= u_\star-\varphi(0)=Q_Nu_\star $.

Substituting $r =d +\text{Res}$ we obtain
\[
r(t) = e^{At}d(0) + \int_{0}^{t}e^{A(t-s)} [F(r(s)+\varphi(s)-\text{Res}(s))-F(\varphi(s))]ds,
\]
which means $r$ is the solution of the following differential equation
\[
\partial_t r = Ar +F(r+\varphi-\text{Res})-F(\varphi).
\]
Recall that $\text{Res}(0)= 0$ so $r(0)=d(0)=Q_Nu_\star$.

Now we use standard a-priori estimates for the equation for $r$. This yields good estimates,
as both the linear part and the nonlinear part are stable, which simplifies the error estimate significantly.
\begin{equation}
\label{e:ap}
 \frac 12 \frac{\partial}{\partial t}\|r\|^2_{L_2}=\langle Ar,r \rangle-\langle(r+\varphi - \text{Res})^3 - \varphi^3,r\rangle.
\end{equation}
The following lemma is necessary to bound the cubic. It is not optimal, but sufficient for our purposes.
\begin{lemma}
	\label{lem:cubic}
	For all $r,R,\varphi\in\mathbb{R}$ we have
	\[
	[-(r+R+\varphi)^3+\varphi^3]r \leq R^4+3R^2\varphi^2
	\]
\end{lemma}

\begin{proof}
	First
	\[
	[-(r+R+\varphi)^3+\varphi^3]r = - 3 r \int_{\varphi}^{r+R+\varphi} \zeta^2 d\zeta = - 3 r \int_{0}^{r+R} (\zeta+\varphi)^2 d\zeta \;.
	\]
	Thus the term is non-positive if $r$ and $r+R$ have the same sign (i.e., in the case $r,r+R\in[0,\infty)$ or $r,r+R\in(-\infty,0]$).
	
	In the remaining two cases we have $|R|\leq|r|$, as for $r \leq 0\leq R+r$ we have $R\geq -r \geq 0$ and for  $r+R \leq 0\leq r$ we have $0\leq r \leq -R$.
	Thus we obtain using $ab\leq a^2+\frac14 b^2$
	\[
	\begin{split}
	[-(r+R+\varphi)^3+\varphi^3]r =& -r^4-3r^3(R+\varphi)-3r^2(R+\varphi)^2-r[(R+\varphi)^3-\varphi^3]\\
	&\leq -\frac34  r^2(R+\varphi)^2-r[(R+\varphi)^3-\varphi^3]\\
	&= -\frac34  r^2(R+\varphi)^2 - rR^3 -3rR\varphi(R+\varphi)\\
	&\leq - rR^3 +3 R^2\varphi^2\\
	& \leq  R^4 +3 R^2\varphi^2 \;.
	\end{split}
	\]
\end{proof}

Now by Lemma \ref{lem:cubic} we obtain from \eqref{e:ap}

\[
\frac{\partial}{\partial t} \|r\|^2_{L_2}
\leq 2\|\text{Res}\|_{L^4}^4 + 6 \|\text{Res}\|_{L^4}^2 \|\varphi\|_{L^4}^2,
\]
and integration yields
\[
\|r(t) \|^2_{L_2}
\leq \|r(0)\|_{L^2}^2 + 2 \int_0^t\|\text{Res}\|_{L^4}^4 ds + 6\Big(\int_0^t \|\text{Res}\|_{L^4}^4 ds \Big)^{1/2}
\Big(\int_0^t  \|\varphi\|_{L^4}^4 ds \Big)^{1/2}
\]
From Theorem \ref{thm:mainres}  we can get the following bound for the error (by using Cauchy-Schwarz) in case $t\in[mh,(m+1)h]$
\[
\mathbb{E}[\|r(t) \|^2_{L_2} |(Z_k)_{k\in\mathbb{N}}]  \leq \|Q_N u_\star\|^2 + 2 (\mathcal{K}_m)^4
+ 6(\mathcal{K}_m)^2  \Big(\int_0^t \|\varphi\|_{L^4}^4 ds \Big)^{1/2}
\]
In order to have a fully numerically computable quantity, we need to take care of the integral. 
We proceed similarly to $I(\tau)$ and use the equality
$ \varphi(t) = u_{n(t)} + \tau(t) h^{-1} d_{n(s)+1}$ to obtain:
\[
\begin{split}
 \int_0^t \|\varphi\|_{L^4}^4 ds
 \leq &   \int_0^t \| u_{n(s)} + \tau(s)h^{-1} d_{n(s)+1}\|_{L^4}^4 ds \\
  = &  \sum_{n=0}^{m(t)} \int_0^h \| u_{n(s)} + s h^{-1} d_{n(s)+1}\|_{L^4}^4 ds \\
 \leq &  h \sum_{n=0}^{m(t)} \Big[  \| u_{n}\|_{L^4}^4
 + 2\| u_{n}\|_{L^4}^3 \|d_{n+1}\|_{L^4}
 + 2\| u_{n}\|_{L^4}^2 \|d_{n+1}\|_{L^4}
 \\ & \qquad\qquad
  +  \| u_{n}\|_{L^4}^3 \|d_{n+1}\|^3_{L^4}+ \frac15\|d_{n+1}\|^4_{L^4}  \Big] \\
\end{split}
\]

\begin{theorem}\label{thm:main}
Let $u$ be a mild solution of \eqref{e:mild} with initial condition $u_\star$,
$\varphi$ the numerical approximation from Section \ref{sec:defphi},
and $\text{Res}$ the numerical approximation from \eqref{e:defRes}.
For $t\in[mh,(m+1)h)$ and $m\in\{1,\ldots,M\}$ with $M=T/h$ we have for the error $r=u-\varphi+\mathrm{Res}$ that
\[
\mathbb{E} \left[\|r(t) \|^2_{L_2}| (Z_k)_{k\in\mathbb{N}}\right]
 \leq \|Q_N u_\star\|^2 + 2 (\mathcal{K}_m)^4
+ 6 (\mathcal{K}_m)^2 ( \mathcal{I}_m )^{1/2}
\]
where the bound on the residual $\mathcal{K}_m$ was defined in Theorem \ref{thm:mainres} and
\[
\begin{split}
\mathcal{I}_m =
 h \sum_{n=0}^{m} \Big[  \| u_{n}\|_{L^4}^4
 + 2\| u_{n}\|_{L^4}^3 \|d_{n+1}\|_{L^4}
 &+ 2\| u_{n}\|_{L^4}^2 \|d_{n+1}\|_{L^4}
 \\ &
  +  \| u_{n}\|_{L^4}^3 \|d_{n+1}\|^3_{L^4}+ \frac15\|d_{n+1}\|^4_{L^4}  \Big]
\end{split}
\].
\end{theorem}
As we expect $\mathcal{K}_m$ to be small, and the solution of the numerical scheme not,
the third term should dominate in the error estimate. This is also confirmed in the numerical example.

Note that we usually neglect the error term coming from the initial condition by assuming that $Q_Nu_\star=0$.
Anyway this can be made as small as we wish, by assuming that the initial condition is sufficiently smooth and choosing $N$ large.

Let us also remark that $r$ is not the error $d=u-\varphi$ we are interested in, but as we expect $\text{Res}$ to be small,
we neglect this in the discussion of the numerical examples later.

\section{Extensions of the Result}
\label{sec:out}

In this section we discuss a few possible generalizations of our result.
As the precise bounds and constants (especially in the approximation result)
depend heavily on the structure of the given equation,
we presented only the Allen-Cahn equation as an example.
Although may methods of the  proof (especially the results for the residual)
should be straightforwardly adapted to other equations.

\noindent\textbf{Stable polynomial nonlinearities:} These should be easy to treat with our results.
The estimate for the residual only contains more terms, if the nonlinearity is of higher order.
But due to the approximation result we have to change all the estimates to $L^{p+1}$ and thus to $L^{(p+1)p}$
if the nonlinearity is a polynomial of  odd degree $p$,
instead of $L^4$ and $L^{12}$ used for Allen-Cahn.

\noindent\textbf{Globally Lipschitz nonlinearities:} It is possible to adapt our result to this case,
but it is not useful, as the analytic error estimate for the numerical approximation that are already available
are quite good.

\noindent\textbf{General differential operators:}  If they are diagonal w.r.t.\ the Fourier basis, then all estimates needed should be
similar. But for general operators none of our estimates for the stochastic convolution and the Bridge process
in between discretization points apply directly. These need to be rewritten in such a case.

\noindent\textbf{General noise:} This possible to treat for additive noise, but it should be more complicated.
Various constants in our estimates are not that easy to compute explicitly for general noise. Moreover,
the generation of $P_NZ(hk)$ is significantly more involved in the numerical scheme if the covariance of the noise is not diagonal in Fourier space.
Thus we restrict ourself in the examples to space-time white noise.

Let us finally remark, that the analysis depends crucially on the spectral Galerkin approximation, that simplifies all estimates a lot.
For other numerical methods like finite element methods, the results for the residual has to be rewritten completely.
Nevertheless the approximation result in the end does not depend on the numerical method but only on the structure of the equation.

\section{Numerical Experiments}
\label{sec:Num}

For the numerical result we focus on space-time white noise of strength $1$, which means that  all $\alpha_k=1$. Moreover,
as both the linear part and the nonlinearity are stable, we expect the solution to be of order $1$ with rare events,
where the solution is significantly larger.
Nevertheless, we expect solutions to be quite rough.

Due to poor regularity properties, we do not expect the numerical approximation to be very accurate,
but still we first tried a relatively poor discretization with $N=128$ Fourier-modes and time-steps $h=10^{-4}$
with a terminal time $T=1$. As expected this did not work that well, 
and the error is only a little bit smaller than the solution,
see Figure \ref{fig:new}. Thus we used in our example
\[
N=256\quad\text{and}\quad h=10^{-6}\;.
\]
To simplify the example a little bit, we consider the initial condition $u_0=u_\star = \sin(x) $
so that the projection to the high modes $Q_Nu_\star=0$ vanishes,
and we can neglect all error terms arising from the initial condition.

First in Figure \ref{fig:KM+Error} we plotted the residual $\mathcal{K}_m^4$ for $m=1,\cdots,1/h$ together with the final error.
As expected $\mathcal{K}_m^4$ is small and the error term from Theorem \ref{thm:main}
is bounded by the error term involving $\mathcal{K}_m^2$ and the numerical data.

In Figure \ref{fig:comparison} we plot two terms of the residual $\mathcal{K}_m^4$ , for $m=1,\cdots,1/h$.
One of the main terms in $\mathcal{K}_m^4$ which depends on the numerical data is $\text{Res}_n^{dat}$,
therefore we plot $h\sum_{n=1}^{m}\|\text{Res}_n^{dat}\|_{L_4}^4 $ in Figure \ref{fig:Residual}
to see impact of these terms on the residual-bound $\mathcal{K}_m^4$, which seems to be negligible.

\begin{figure}[ht]
       \subfigure[$\mathcal{K}_m^4$]
    {
        \includegraphics[width=2.7in]{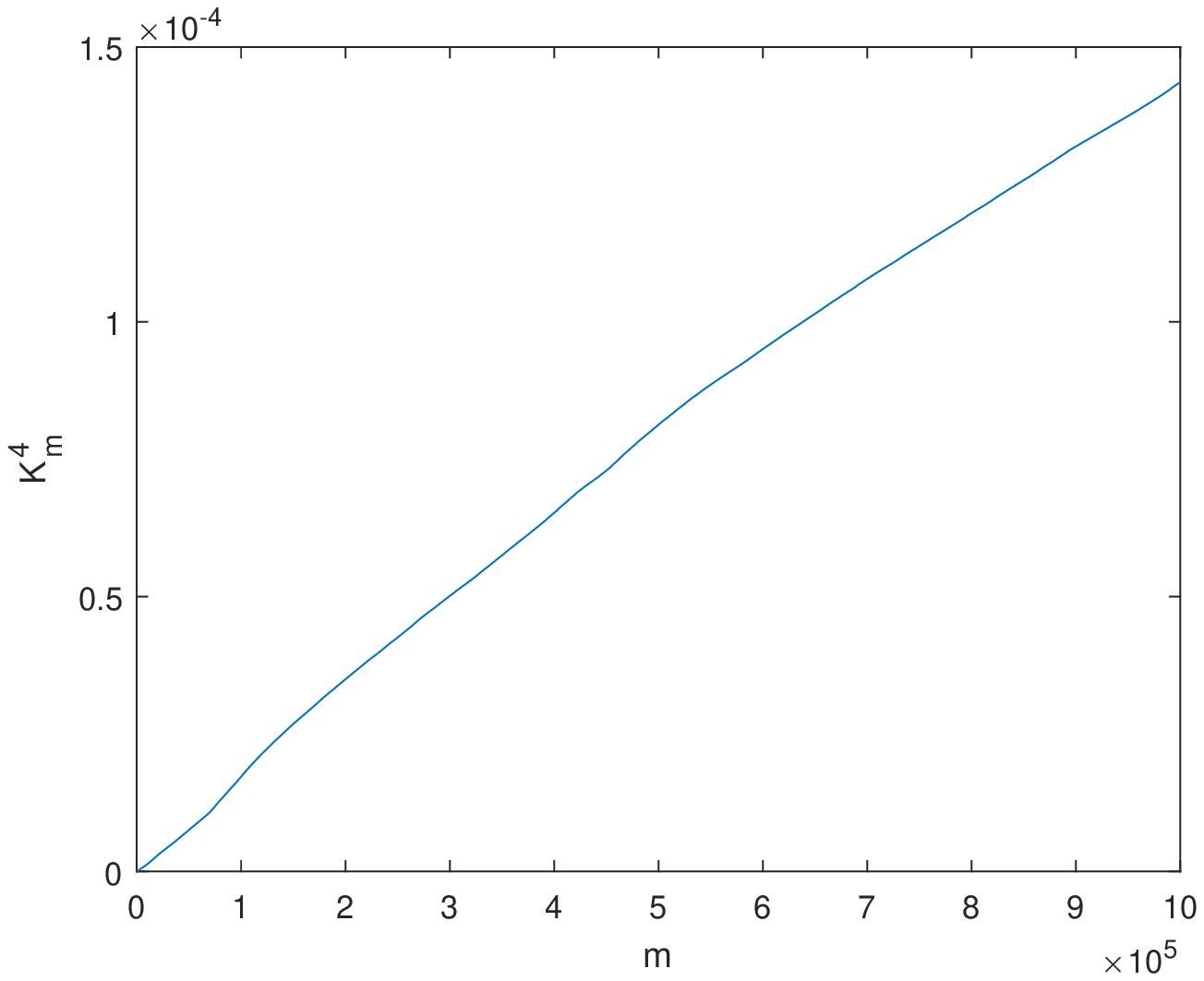}
        \label{fig:Km}
    }
     \subfigure[$\mathbb{E} \Big(\|r(t)\|^2_{L_2} | (Z_k)_{k\in\mathbb{N}}\Big)$]
    {
        \includegraphics[width=2.7in]{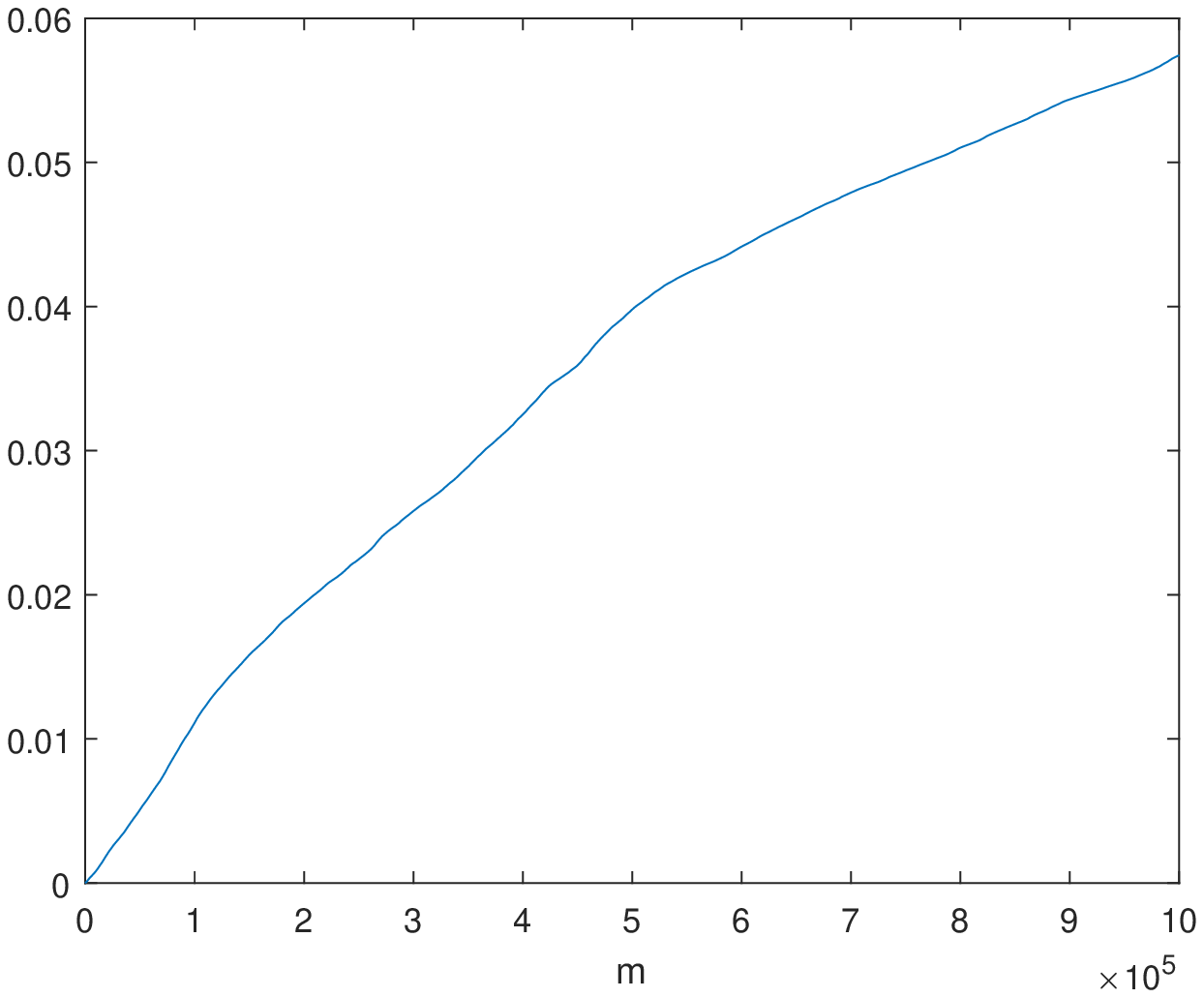}
        \label{fig:FinalError}
    }
    \caption{Comparison of the bound $\mathcal{K}_m^4$ on the residual and the final error bound from  Theorem \ref{thm:main} with $N=256, h=10^{-6}$. Obviously, the $\mathcal{K}_m^4$ is not relevant in that estimate.}
    \label{fig:KM+Error}
\end{figure}

Moreover in Figure \ref{fig:OU-Bridge} we plot $h \sum_{n=0}^{m} \mathcal{S}_h(Z_{n+1}-e^{hA}Z_n)$,
i.e, the term in $\mathcal{K}_m^4$ which arises from the OU-bridge.
By comparing Figure \ref{fig:Km} and \ref{fig:OU-Bridge} we can see the impact of the OU-bridge on $\mathcal{K}_m^4$.
This gives a substantial, but not the most dominant term in $ \mathcal{K}_m^4$.
We can also see that this error term is almost growing linear.
The reason for this is that the part in $\mathcal{S}_h$ that depends on the numerical data $Z_n$ is quite small
and the deterministic part of the estimate dominates, which bounds the fluctuations of the OU-bridge
between the data points.

\begin{figure}[ht]
       \subfigure[$h \sum_{n=0}^{m} \mathcal{S}_h(Z_{n+1}-e^{hA}Z_n)$]
    {
        \includegraphics[width=2.7in]{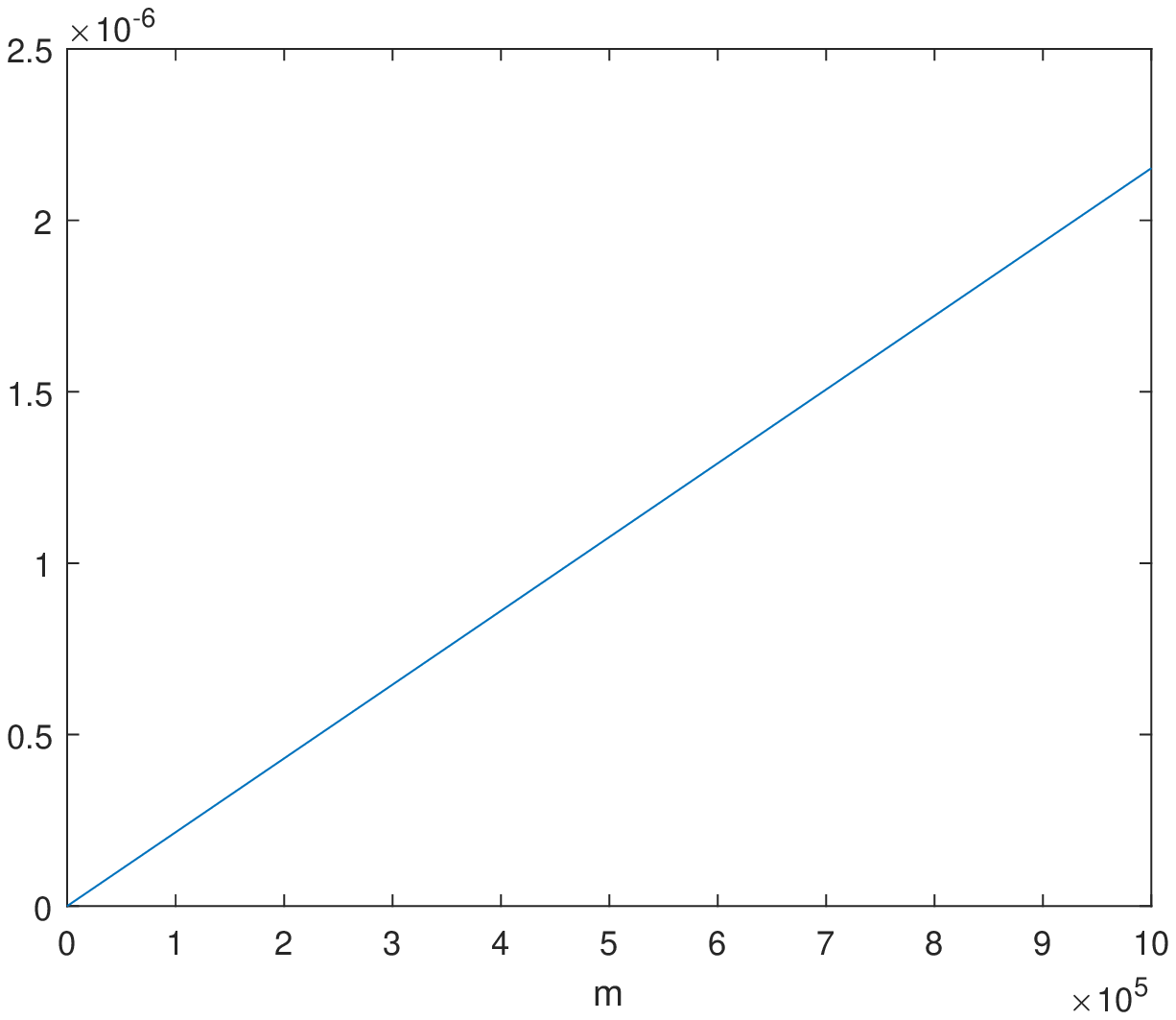}
        \label{fig:OU-Bridge}
    }
        \subfigure[$h\sum_{n=1}^{m} \|\text{Res}^{\text{dat}}_n\|_{L_4}^4$]
    {
        \includegraphics[width=2.7in]{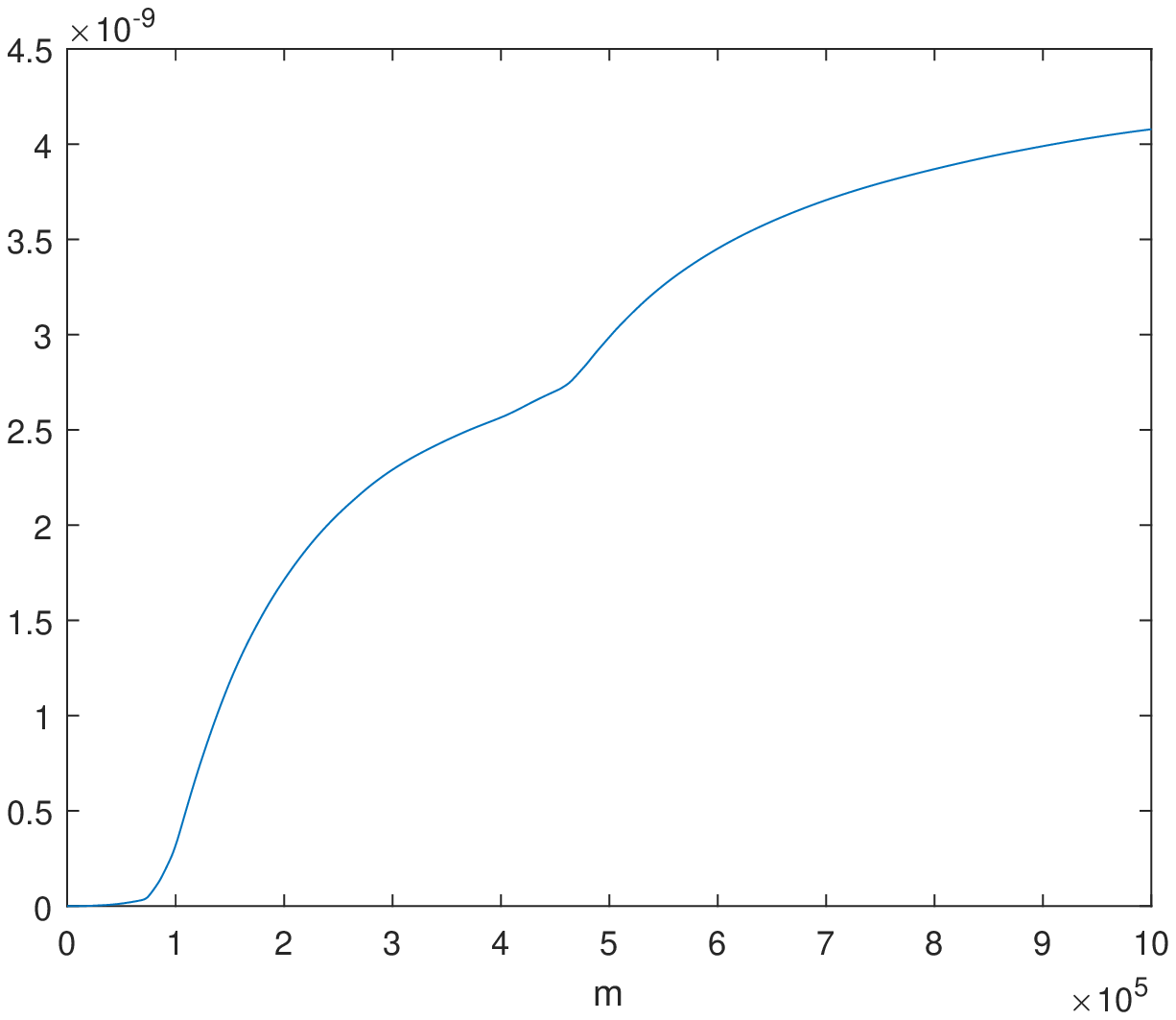}
        \label{fig:Residual}
    }
    \caption{Values of $h \sum_{n=0}^{m} \mathcal{S}_h(Z_{n+1}-e^{hA}Z_n)$ which bounds the OU-bridge. This gives a substantial, but not the most dominant term in $ \mathcal{K}_m^4$.
     The data dependent terms $h\sum_{n=1}^{m} \|\text{Res}^{\text{dat}}_n\|_{L_4}^4$  at the discrete time-points in the residual are negligible.
    Occasionally these terms become suddenly larger, at points where we have a stronger increase in the error.}
    \label{fig:comparison}
\end{figure}

The final bound for the error $\mathbb{E} \left[\|r(mh) \|^2_{L_2}| (Z_k)_{k\in\mathbb{N}}\right]$ which is stated in Theorem \ref{thm:main}  is plotted in Figure \ref{fig:Average error}
for $10$ simulations. It confirms that the numerical approximation with $N=256$ and $h=10^{-6}$ works well,
in contrast to the case $N=128$ and $h=10^{-4}$. See Figure \ref{fig:new}.

We also see in Figure \ref{fig:Average error} and even better in Figure \ref{fig:new}
that the error is not growing with constant speed,
but it has parts where it grows much faster. This effect is also very well visible in Figure \ref{fig:Residual},
although the effect there is too small to have an impact on $\mathcal{K}_m$.
We conjecture that this might be a large deviation effect, that actually might not be that rare due to noise strength of order one.

Let us also point out that we do not expect to have a mean-square error bound without conditioning on the numerical data.
Thus both in Figure  \ref{fig:Average error}
and \ref{fig:new} we expected a quite large variation for different realizations of the numerical approximation.

To see exactly the impact of each term in  $\mathcal{K}_m$ we plotted its value in Figure \ref{fig:Km Terms}.
Also in Table \ref{t1} values of each term at the final time $T=1$ is stated for $4$ simulations.

\begin{figure}[ht]
\begin{center}
  \includegraphics[width=4in]{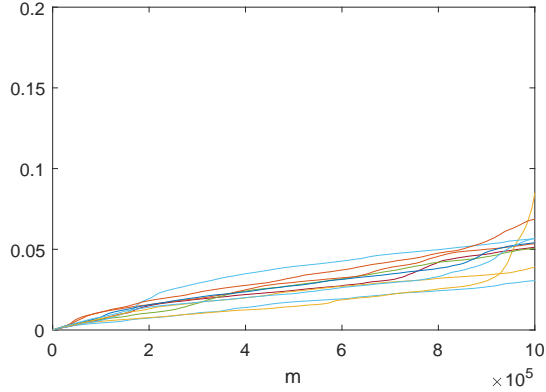}
\end{center}
            \caption{$10$ simulations of our bound for $\mathbb{E} \Big(\|r(t)\|^2_{L_2} | (Z_k)_{k\in\mathbb{N}}\Big)$ for $N=256$ and $h=10^{-6}$.}
                \label{fig:Average error}
 \end{figure}

\begin{figure}[ht]
\begin{center}
  \includegraphics[width=4in]{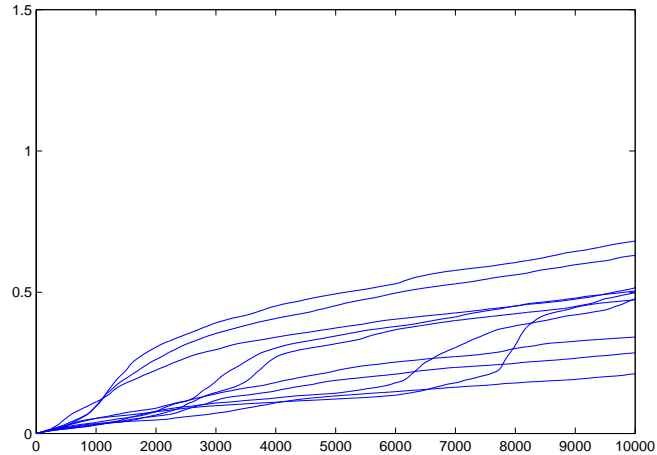}
\end{center}
            \caption{$10$ simulations of our bound for $\mathbb{E} \Big(\|r(t)\|^2_{L_2} | (Z_k)_{k\in\mathbb{N}}\Big)$ for $N=128$ and $h=10^{-4}$.
            Only for small times the error estimate seems reasonable. Moreover, the data has quite a variance.}
                \label{fig:new}
 \end{figure}
 \begin{figure}[ht]
\begin{center}
  \includegraphics[width=5in]{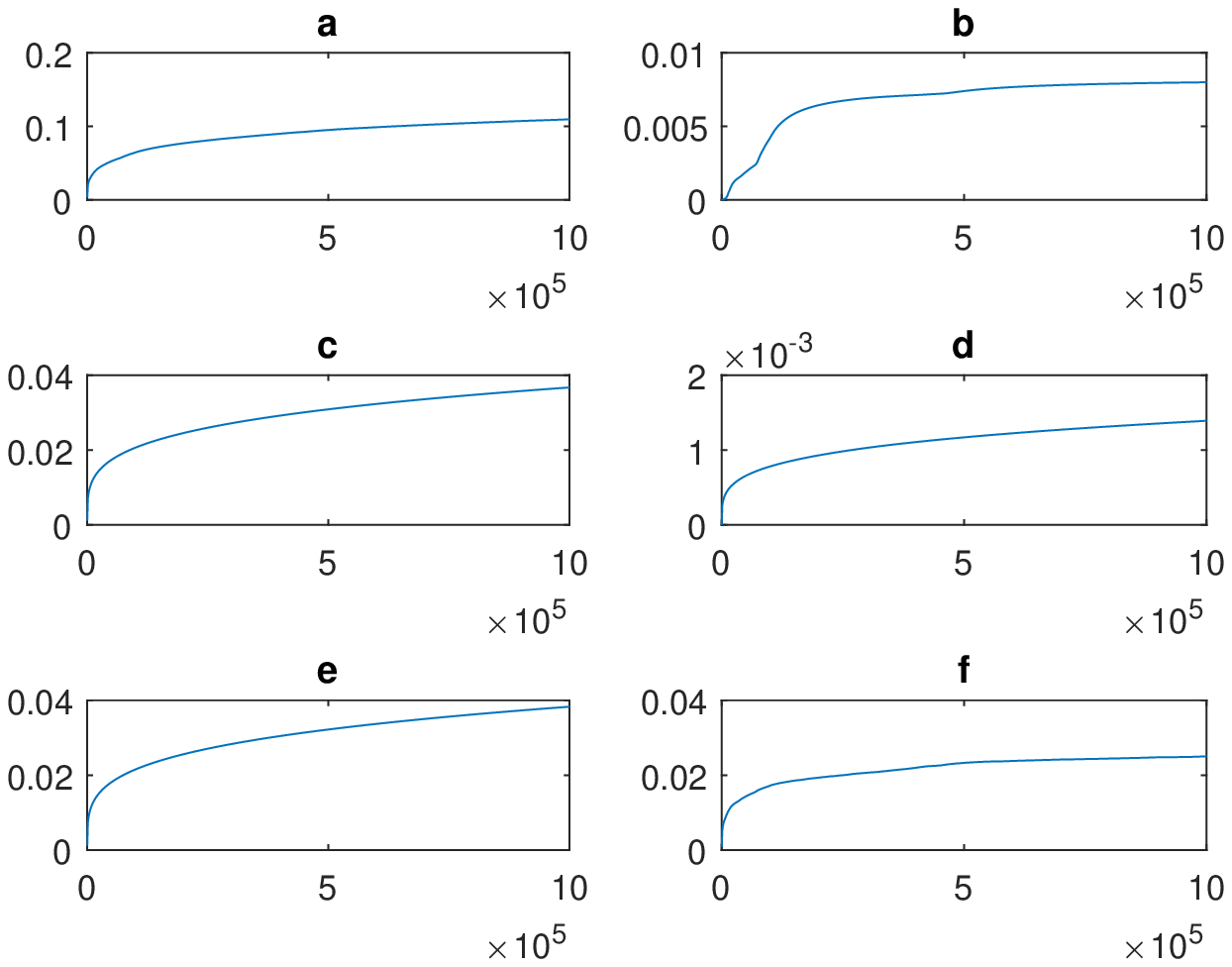}
\end{center}
                 \caption{Values of the error terms for one realization  
                 $(a)\ \mathcal{K}_m$,  
                 $(b)\ \Big(h\sum_{n=1}^{m} \|\text{Res}^{\text{dat}}_n\|_{L_4}^4  \Big)^{1/4},$  
                 $(c)\ \Big(\frac{3mh}{8\pi}\Big)^{1/4} \Big( \sum_{k>N} \frac{\alpha_k^2}{k^2}\Big)^{1/2},$ 
                 $(d)\ \Big( \frac{h^5}{5}  \sum_{n=0}^{m} \|Au_{n}\|_{L^4}^4   \Big)^{1/4} $,
                 $(e)\ \Big( h^2 \sum_{n=0}^{m} \Big[
\frac12 \|u_n\|_{L^{12}}^3
 + \cdots+ \frac1{20} \|d_{n+1}\|_{L^{12}}^3\Big]\Big)^{1/4}$,
                    $(f)\  \Big(h \sum_{n=0}^{m} \mathcal{S}_h(Z_{n+1}-e^{hA}Z_n)\Big)^{1/4}$
           }
                  \label{fig:Km Terms}
 \end{figure}


\begin{table}[!htp]\label{}
\center
\begin{tabular}{ | c | c | c | c | c | }
\hline $\mathcal{K}_m$ &$ 0.1027$ & $ 0.0996$ & $0.1055$ & $0.1085$  \\
\hline $E_1$ &$0.0030$ & $0.0012$ & $0.0038$ & $0.0076$  \\
\hline $E_2$ &$  0.0367$ & $0.0367$ & $0.0367$ & $0.0367$  \\
\hline $E_3$  &$0.0014$ & $0.00141$ & $0.0014$ & $0.00141$  \\
\hline $E_4$ &$0.0233$ & $0.0220$ & $ 0.0253$ & $0.0246$  \\
\hline $E_5$ &$0.03831$ & $0.0383$ & $0.03830$ & $0.0383$  \\
\hline
\end{tabular}
\protect\caption{Values of four simulations of $\mathcal{K}_m$ and each of it's term at the final time $T=1$, i.e, $m=10^6$.
The contribution of error at discretization points is $ E_1=\Big(h\sum_{n=1}^{m} \|\text{Res}^{\text{dat}}_n\|_{L_4}^4  \Big)^{1/4}$
for the part determined by the data
and  $E_2 = \Big(\frac{3mh}{8\pi}\Big)^{1/4} \Big( \sum_{k>N} \frac{\alpha_k^2}{k^2}  \Big)^{1/2}$ for the stochastics on the high modes.
The next two error terms $ E_3=\Big( \frac{h^5}{5}  \sum_{n=0}^{m} \|Au_{n}\|_{L^4}^4 ds  \Big)^{1/4} $ and $ E_4=\Big( h^2 \sum_{n=0}^{m} \Big[
\frac12 \|u_n\|_{L^{12}}^3
 + \cdots+ \frac1{20} \|d_{n+1}\|_{L^{12}}^3\Big]   \Big)^{1/4}$ are data controlled terms for error that arises between discretization points.
Finally  $ E_5=\Big(h \sum_{n=0}^{m} \mathcal{S}_h(Z_{n+1}-e^{hA}Z_n)\Big)^{1/4}$ bounding the stochastic fluctuation
in between discretization points is large, but not that large as we would expected it to be.
}\label{t1}
\end{table}

\end{document}